\documentclass[a4paper,reqno]{amsart}
\usepackage[a4paper,hscale=0.72,vscale=0.72,centering]{geometry}

\newcommand{\erre}{\mathbb{R}}
\newcommand{\E}{\mathbb{E}}

\renewcommand{\P}{\mathbb{P}}
\newcommand{\m}{\bar{\mu}}

\newcommand{\dom}{\mathop{\mathrm{dom}}\nolimits}
\newcommand{\ip}[2]{\langle #1,#2 \rangle}
\newcommand{\bip}[2]{\left\langle #1,#2 \right\rangle}

\newsymbol\lesssim 132E

\newtheorem{prop}{Proposition}
\newtheorem{thm}[prop]{Theorem}

\newtheorem{defi}[prop]{Definition}
\theoremstyle{remark}
\newtheorem{rmk}[prop]{Remark}

\begin{document}
\title[Uniqueness of mild solutions]{On uniqueness of mild solutions for
  dissipative stochastic evolution equations}

\author{Carlo Marinelli}
\address[C.~Marinelli]{Facolt\`a di Economia, Universit\`a di Bolzano,
  Piazza Universit\`a 1, I-39100 Bolzano, Italy, and Dipartimento di
  Matematica, Universit\`a di Trento, I-38123 Trento, Italy.}
\urladdr{http://www.uni-bonn.de/$\sim$cm788}

\author{Michael R\"ockner}
\address[M.~R\"ockner]{Fakult\"at f\"ur Mathematik, Universit\"at
Bielefeld, Postfach 100 131, D-33501 Bielefeld, Germany.}
\email{roeckner@math.uni-bielefeld.de}

\date{21 January 2010}

\begin{abstract}
  In the semigroup approach to stochastic evolution equations, the
  fundamental issue of uniqueness of mild solutions is often
  ``reduced'' to the much easier problem of proving uniqueness for
  strong solutions. This reduction is usually carried out in a formal
  way, without really justifying why and how one can do that. We
  provide sufficient conditions for uniqueness of mild solutions to a
  broad class of semilinear stochastic evolution equations with
  coefficients satisfying a monotonicity assumption.
\end{abstract}



\thanks{We are grateful to Viorel Barbu for several helpful
  conversations. A large part of the work for this paper was carried
  out while the first author was visiting the Department of
  Mathematics of Purdue University supported by a fellowship of the
  EU}

\maketitle


\section{Introduction}
The purpose of this work is to prove uniqueness of mild and weak (in
the PDE sense) solutions to dissipative stochastic evolution equations
of the type
\begin{equation}     \label{eq:zero}
du(t) + Au(t)\,dt + Fu(t)\,dt = B(t,u(t))\,dW(t)
+ \int_Z G(t,u(t-),z)\,\bar{\mu}(dt,dz)
\end{equation}
on a real Hilbert space $H$, where $W$ is a Wiener process and $\m$ a
compensated Poisson measure (precise definitions and assumptions will
be given below).

The motivation for this work is that we have not been able to
understand the arguments, often very concise, used in the
literature. It is probably worth elaborating more on this observation,
as most readers will be surprised that we return to such a basic issue
as uniqueness for solutions of equations that are indeed expected to
be well-posed thanks to their dissipative nature. Essentially all
proofs of uniqueness that we have been able to find go as follows:
suppose that (\ref{eq:zero}) has two solutions, and further assume
they are strong. Then an application of It\^o's formula for the square
of the norm, monotonicity, and Gronwall's lemma quickly yield that the
two strong solutions must coincide. Now comes the trouble: if the
solutions are not strong, consider a ``suitable'' regularization. The
problem is that mild solutions, except in the non-interesting cases of
Lipschitz nonlinearities, are constructed using regularizations,
usually of $A$ and $F$. Therefore, without fully elaborating the
argument (something that is not done in the literature we know of),
one could at most prove uniqueness of mild solutions constructed by
regularization, and thus it seems like one is trapped in a vicious
circle. In fact, there is no guarantee that other mild solutions could
be constructed without resorting to regularized equations and limit
passages. We should also clarify that we are not claiming that the
literature contains errors, but it is probably fair to say that the
usual arguments are a bit mysterious and perhaps not fully convincing.

After we realized we could not easily understand how to prove
uniqueness by the ``simple'' regularization procedure alluded to in
the literature, each of us obtained an independent proof, by different
arguments. The two proofs are collected in the present paper.

\section{Preliminaries and notation} \label{sec:prel} Let $T>0$ and
$(\Omega,\mathcal{F},\mathbf{F},\P)$,
$\mathbf{F}=\{\mathcal{F}_t\}_{0\leq t\leq T}$ a fixed filtered
probability space (satisfying the ``usual'' assumptions), on which all
random elements will be defined. The predictable sigma-field on this
stochatic basis will be denoted by $\mathcal{P}$. Let $K$, $H$ be real
separable Hilbert spaces, $\mathcal{L}_2(K \to H)$ the set of
Hilbert-Schmidt operators from $K$ to $H$. By $W$ we shall denote a
$K$-valued Wiener process with covariance operator $Q$, while $\mu$
will denote a Poisson measure on $[0,T] \times Z$ with compensator
$\mathrm{Leb}\otimes m$, where $(Z,\mathcal{Z},m)$ is a measure space,
and $\m:=\mu-\mathrm{Leb} \otimes m$ stands for the compensated
measure associated to $\mu$. Here and in the following we shall denote
Lebesgue measure by $\mathrm{Leb}$.

Let us recall a few facts about stochastic integration with respect to
Wiener processes and compensated Poisson measures. For all unexplained
(but classical) results and notations we refer to \cite{Met} (cf. also
\cite{Sko-studies}). Let $\mathfrak{I}_Q$ denote the set of all
progressively measurable processes $\Phi:[0,T] \to \mathcal{L}_2^Q(K
\to H)$ such that
\[
\P\Big( \int_0^T |\Phi(t)|_Q^2\,dt < \infty \Big) = 1,
\]
where $\mathcal{L}_2^Q(K \to H)$ denotes the space of
linear (possibly unbounded) operators from $K$ to $H$ that belong to
$\mathcal{L}_2(Q^{1/2}K \to H)$, endowed with the norm
\[
|\cdot|_Q := |\cdot|_{\mathcal{L}_2(Q^{1/2}K \to H)}.
\]
Then for any $F \in \mathfrak{I}_Q$ the stochastic integral $(F\cdot
W)_t:=\int_0^t F(s)\,dW(s)$ is well-defined for all $t \leq T$ and $F
\cdot W$ is a local martingale. Similarly, denoting the set of all
(random) functions $\phi:[0,T] \times Z \to H$ that are $\mathcal{P}
\otimes \mathcal{Z}$-measurable and satisfy
\[
\P\Big( \int_0^T\!\int_Z |\phi(s,z)|^2\,m(dz)\,ds < \infty \Big) = 1
\]
by $\mathfrak{I}_m$, we have that, for any $g \in \mathfrak{I}_m$, the
stochastic integral
\[
(g \star \m)_t := \int_{]0,t]}\!\int_Z g(s,z) \,\m(ds,dz)
\]
is well-defined for all $t \leq T$ and $g \star \m$ is a local
martingale.  Moreover, if $F_n(s) \to F(s)$ in probability for
a.a. $s$ and there exists $\Phi \in \mathfrak{I}_Q$ such that
$|F_n(s)|_Q \leq |\Phi(s)|_Q$ $\P$-a.s. for a.a. $s$, then $(F_n \cdot
W)_t \to (F \cdot W)_t$ in probability for all $t$. Similarly, if
$g_n(s,z) \to g(s,z)$ in probability for $\mathrm{Leb}\otimes
m$-a.a. $(s,z)$ and there exists $\phi \in \mathfrak{I}_m$ such that
$|g_n(s,z)| \leq \phi(s,z)$ $\P$-a.s. for $\mathrm{Leb}\otimes
m$-a.a. $(s,z)$, then $(g_n \star \m)_t \to (g \star \m)_t$ in
probability for all $t$. For simplicity of notation we shall write
$\int_0^t$ instead of $\int_{]0,t]}$ when integrating against random
measures, and we shall denote the norm in $L^2(Z,m)$ by $|\cdot|_m$.
Finally, we recall that $u$ is a mild solution to (\ref{eq:zero}) with
initial condition $u(0)=u_0$ if one has
\begin{multline*}
u(t) + \int_0^t e^{-(t-s)A} Fu(s)\,ds \\= e^{-tA}u_0
+ \int_0^t e^{-(t-s)A} B(s,u(s))\,dW(s) 
+ \int_0^t\!\int_Z e^{-(t-s)A} G(s,u(s-),z)\,\m(ds,dz)
\end{multline*}
$\P$-a.s. for all $t \leq T$ (and all integrals are well-defined).


\section{Uniqueness of mild solutions by a bootstrap argument}
In this section we give a proof of uniqueness that is somewhat
reminiscent of bootstrap arguments used in deterministic PDE. In the
first subsection we consider the simpler case of equations with
additive noise (i.e. with $B$ and $G$ in (\ref{eq:zero}) independent
of $u$), and in the second subsection we consider the general
case. The first subsection is included both because the argument is
relatively short, and also because it will be used later to prove
uniqueness for a class of equations which is not covered by the
results of this section.

\subsection{Additive noise}
Assuming that the coefficients in front of the Wiener process and the
Poisson measure do not depend explicitly on the unknown, we obtain
uniqueness by a pathwise argument, thus using essentially a
deterministic argument. More precisely, we have the following
\begin{thm}     \label{thm:add}
  Consider the stochastic evolution equation on $H$
  \begin{equation}     \label{eq:add}
  du(t) + Au(t)\,dt + Fu(t)\,dt = B(t)\,dW(t) + \int_Z G(t,z)\,\m(dt,dz),
  \end{equation}
  where $A$ is a linear monotone operator on $H$, $F$ a (nonlinear)
  operator on $H$ such that $x \mapsto Fx+\eta x$ is monotone for some
  $\eta \in \erre$, $B$ is progressively measurable, $G$ is
  $\mathcal{P}\otimes\mathcal{Z}$-measurable, and
  \[
  \int_0^t \big( \big| e^{(t-s)A}B(s) \big|_Q^2 
          + \big|e^{(t-s)A}G(s,\cdot)\big|^2_m\big) \,ds
  < \infty
  \]
  $\P$-a.s. for all $t \leq T$. Then (\ref{eq:add}) admits at most one
  mild solution $u$ such that $|Fu|_{L^1([0,T] \to H)}<\infty$
  $\P$-a.s..
\end{thm}
\begin{proof}
Let $u$, $v$ be two mild solutions of (\ref{eq:add}),
and define
\[
y(s) := u(s)-v(s), \qquad g(s) := Fv(s) - Fu(s).
\]
We shall keep $g$ fixed from now on, and we will work
``$\omega$-by-$\omega$''. In particular, the hypotheses imply the
existence of $\Omega' \subset \Omega$ such that $\P(\Omega')=1$ and $g
\in L^1([0,T] \to H)$ for all $\omega \in \Omega'$. Let us fix $\omega
\in \Omega'$ from now on. Then $y$ is a mild solution of the equation
\[
dy(t) + Ay(t)\,dt = g(t)\,dt
\]
with initial condition $y(0)=0$ and $\sup_{t \leq T}|y(t)|<\infty$. Let
$A_\varepsilon:=A(I+\varepsilon A)^{-1}$ denote the Yosida
approximation of $A$, and let $y_\varepsilon$ be the (unique) strong
solution of the equation obtained by replacing $A$ with
$A_\varepsilon$ in the previous one.  Trotter-Kato's approximation
theorem (see e.g. \cite[p.~241]{barbu}) then yields
\[
\lim_{\varepsilon \to 0} \sup_{t \leq T} |y_\varepsilon(t) - y(t)| = 0.
\]
We also have
\[
\frac12 \frac{d}{dt} |y_\varepsilon(t)|^2
+ \ip{A_\varepsilon y_\varepsilon(t)}{y_\varepsilon(t)}
= \ip{g(t)}{y_\varepsilon(t)}
= \ip{g(t)}{y(t)} + \ip{g(t)}{y_\varepsilon(t)-y(t)},
\]
therefore, by the monotonicity of $A_\varepsilon$ and
\[
-\ip{g(t)}{y(t)} = \ip{Fu(t)-Fv(t)}{u(t)-v(t)} \geq -\eta|u(t)-v(t)|^2,
\]
we get
\begin{align*}
|y_\varepsilon(t)|^2 &\leq \eta \int_0^t |y(s)|^2\,ds 
+ \int_0^t \ip{g(s)}{y_\varepsilon(s)-y(s)}\,ds\\
&\leq \eta \int_0^t |y(s)|^2\,ds
+ \sup_{s \leq t} |y_\varepsilon(s)-y(s)| \int_0^t |g(s)|\,ds.
\end{align*}
Letting $\varepsilon \to 0$, we conclude, recalling that $g \in
L^1([0,T] \to H)$, that $y(t)=u(t)-v(t)=0$ for all $t \leq T$ by an
application of Gronwall's inequality. Since $\omega \in \Omega'$ was
arbitrary, we have $u(t)=v(t)$ $\P$-a.s..
\end{proof}
\begin{rmk}
  Note that the above proof does not use anywhere the linearity of
  $A$. The same observation applies to the method used in the next
  subsection.
\end{rmk}

\subsection{General case}
In this subsection we consider the general case of equations with
multiplicative noise under a monotonicity assumption on $A$ and on the
triplet $(F,B,G)$.

Throughout this subsection we shall always assume, without further
mention, that $B:[0,T] \times H \to \mathcal{L}_2^Q$ and $G:[0,T]
\times H \times Z \to H$ satisfy the usual measurability conditions
needed to ensure that the corresponding stochastic integrals with
respect to $W$ and $\m$ are meaningful.
\begin{thm}     \label{thm:boot}
  Assume that $A$ is a linear monotone operator on $H$, $F$ is a
  (nonlinear) operator on $H$, and $B$, $G$ satisfy the monotonicity
  property
  \[
  2 \ip{Fu-Fv}{u-v}
  - \big| B(s,u)-B(s,v) \big|^2_{Q}
  - \big| G(s,u,\cdot)-G(s,v,\cdot) \big|^2_{m}
  \geq \alpha |u-v|^2
  \]
  for all $u$, $v \in \dom(F)$ and all $s \leq T$, for some $\alpha
  \in \erre$ independent of $s$. Then there is at most one c\`adl\`ag
  mild solution of the stochastic evolution equation (\ref{eq:zero})
  such that
  \begin{equation}     \label{eq:strigne}
  \int_0^T \big( |Fu(s)| + |B(s,u(s))|_Q^2 + |G(s,u(s),\cdot)|_m^2 \big)\,ds 
  < \infty
  \end{equation}
  $\P$-a.s..
\end{thm}
\begin{proof}
  Let $u$, $v$ be two c\`adl\`ag mild solutions of (\ref{eq:zero})
  satisfying condition (\ref{eq:strigne}), and define
  \begin{align*}
    y(s) &:= u(s)-v(s),&            g(s) &:= Fu(s) - Fv(s),\\
    C(s) &:= B(s,u(s))-B(s,v(s)),& D(s,z) &:= G(s,u(s-),z)-G(s,v(s-),z).
  \end{align*}
  We shall keep $g$, $C$, $D$ fixed from now on. Then $y$ is a
  mild solution of the equation
  \[
  dy(t) + Ay(t)\,dt + g(t)\,dt = C(t)\,dW(t) 
  + \int_Z D(t,z)\,\bar{\mu}(dt,dz)
  \]
  with initial condition $y(0)=0$, and $\sup_{t \leq T}|y(t)|<\infty$
  $\P$-a.s. (because $y$ has c\`adl\`ag paths). Let us define
  \begin{align*}
    g_\varepsilon(t)   &:= (I+\varepsilon A)^{-1} g(t),\\
    C_\varepsilon(t)   &:= (I+\varepsilon A)^{-1} C(t),\\
    D_\varepsilon(t,z) &:= (I+\varepsilon A)^{-1} D(t,z).
  \end{align*}
  Then the equation
  \[
  dy(t) + Ay(t)\,dt + g_\varepsilon(t)\,dt = C_\varepsilon(t)\,dW(t)
  + \int_Z D_\varepsilon(t,z)\,\bar{\mu}(dt,dz),
  \]
  with initial condition $y(0)=0$, admits a unique mild solution
  $y_\varepsilon$, which is also a strong solution. One can actually
  immediately verify that $y_\varepsilon(t)=(I+\varepsilon
  A)^{-1}y(t)$, so that, in particular,
  \begin{equation}     \label{eq:ipse}
  \lim_{\varepsilon \to 0} \; \sup_{t \leq T}
  \big| y_\varepsilon(t)-y(t) \big| = 0
  \qquad \P\text{-a.s.}
  \end{equation}
  (we shall omit the indication that statements are meant to hold
  $\P$-a.s. in the rest of the proof, if no confusion may arise).
  It\^o's formula for the square of the norm yields
  \begin{equation} \label{eq:ito}
  \begin{aligned}
  |y_\varepsilon(t)|^2 &+ 2\int_0^t \ip{Ay_\varepsilon(s)}{y_\varepsilon(s)}\,ds
  + 2\int_0^t \ip{g_\varepsilon(s)}{y_\varepsilon(s)}\,ds\\
  &= 2\int_0^t \ip{y_\varepsilon(s)}{C_\varepsilon(s)\,dW(s)} 
  + 2\int_0^t\!\int_Z \ip{y_\varepsilon(s-)}{D_\varepsilon(s,z)}\,\bar{\mu}(ds,dz)\\
  &+ \int_0^t |C_\varepsilon(s)|^2_Q\,ds 
  + \int_0^t\!\int_Z |D_\varepsilon(s,z)|^2\,\mu(ds,dz).
  \end{aligned}
  \end{equation}
  Clearly we have $g_\varepsilon(t) \to g(t)$ for all $t \leq T$ as
  $\varepsilon \to 0$, and
  \begin{gather*}
    \ip{g_\varepsilon(s)}{y_\varepsilon(s)} \leq |g(s)|\,\sup_{s\leq t} |y(s)|,\\
    \int_0^t |g(s)|\,\sup_{s\leq t} |y(s)|\,ds = \sup_{s\leq t} |y(s)|
    \int_0^t |g(s)|\,ds < \infty,
  \end{gather*}
  hence, by the dominated convergence theorem,
  \[
  \lim_{\varepsilon \to 0} \int_0^t
  \ip{g_\varepsilon(s)}{y_\varepsilon(s)}\,ds = \int_0^t
  \ip{g(s)}{y(s)}\,ds.
  \]
  For any $s \leq T$, setting $\tilde{C}(s): K \to \erre$,
  $\tilde{C}(s): \zeta \mapsto \ip{y(s)}{C(s)\zeta}$, and defining
  $\tilde{C}_\varepsilon$ replacing $y$ and $C$ with $y_\varepsilon$
  and $C_\varepsilon$, respectively, we have
  $|\tilde{C}_\varepsilon(s) - \tilde{C}(s)|_Q \to 0$ in
  probability. By the inequality
  \[
  |\tilde{C}_\varepsilon(s)|_Q \leq |y_\varepsilon(s)| \,
  |C_\varepsilon(s)|_Q \leq |y(s)| \, |C(s)|_Q
  \]
  we infer
  \[
  \int_0^T |\tilde{C}_\varepsilon(s)|^2_Q\,ds \leq \sup_{s\leq T}
  |y(s)|^2 \int_0^T |C(s)|_Q^2 < \infty.
  \]
  Since the above bounds are uniform with respect to $\varepsilon$, we
  immediately deduce that we can apply the convergence results for
  stochastic integrals mentioned in Section \ref{sec:prel}, obtaining
  \[
  \int_0^t \ip{y_\varepsilon(s)}{C_\varepsilon(s)\,dW(s)} = \int_0^t
  \tilde{C}_\varepsilon(s)\,dW(s) \to \int_0^t \tilde{C}(s)\,dW(s) =
  \int_0^t \ip{y(s)}{C(s)\,dW(s)}
  \]
  in probability for all $t$.  An analogous argument proves that we
  have
  \[
  \int_0^t\!\int_Z
  \ip{y_\varepsilon(s-)}{D_\varepsilon(s,z)}\,\bar{\mu}(ds,dz) \to
  \int_0^t\!\int_Z \ip{y(s-)}{D(s,z)}\,\bar{\mu}(ds,dz)
  \]
  in probability for all $t$.

  Passing to the limit as $\varepsilon \to 0$ in (\ref{eq:ito}) we are
  left with
  \[
  |y(t)|^2 + 2\int_0^t \ip{g(s)}{y(s)}\,ds \leq M(t) + \int_0^t
  |C(s)|^2_Q\,ds + \int_0^t\!\int_Z |D(s,z)|^2\,\mu(ds,dz),
  \]
  where $M$ is the local martingale defined by
  \[
  M(t) = 2 \int_0^t \ip{y(s)}{C(s)\,dW(s)} + 2\int_0^t \! \int_Z
  \ip{y(s-)}{D(s,z)}\,\bar{\mu}(ds,dz).
  \]
  Let us define the sequences of stopping times
  \begin{align*}
    \tau^1_n &:= \inf\Big\{ t \geq 0:\;\int_0^t |C(s)|_Q^2\,ds \geq n \Big\},\\
    \tau^2_n &:= \inf\Big\{ t \geq 0:\;\int_0^t |D(s,\cdot)|_m^2\,ds
    \geq n \Big\},
  \end{align*}
  and note that both are predictable (as hitting times of continuous
  adapted processes). Then $\tau_n:=\tau^1_n \wedge \tau^2_n$ is
  easily seen to be a localizing sequence of stopping times for $M$,
  so that we get
  \begin{multline*}
    \E |y(\tau_n \wedge t)|^2 + 2\E\int_0^{\tau_n \wedge t} \ip{g(s)}{y(s)}\,ds\\
    \leq \E\int_0^{\tau_n \wedge t} |C(s)|^2_Q\,ds + \E\int_0^{\tau_n
      \wedge t}\!\int_Z |D(s,z)|^2\,\mu(ds,dz).
  \end{multline*}
  We also have, recalling that $\mathrm{Leb} \otimes m$ is the
  compensator of $\mu$,
  \begin{align*}
    \E\int_0^{\tau_n \wedge t}\!\int_Z |D(s,z)|^2\,\mu(ds,dz)
    &= \E\int_0^t \! \int_Z |D(s,z)|^2\mathbf{1}_{s\leq\tau_n}\,\mu(ds,dz)\\
    &= \E\int_0^t \! \int_Z |D(s,z)|^2\mathbf{1}_{s\leq\tau_n}\,m(dz)\,ds\\
    &= \E\int_0^{\tau_n \wedge t}\!\int_Z |D(s,z)|^2\,m(dz)\,ds,
  \end{align*}
  hence
  \[
  \E |y(\tau_n \wedge t)|^2 + \E\int_0^{\tau_n \wedge t} \big(
  2\ip{g(s)}{y(s)} - |C(s)|^2_Q - |D(s,\cdot)|_m^2\big)\,ds \leq 0,
  \]
  and, by the monotonicity assumption on the coefficients,
  \[
  0 \geq \E |y(\tau_n \wedge t)|^2 + 2\alpha \E\int_0^{\tau_n \wedge t}
  |y(s)|^2\,ds
  = \E |y(\tau_n \wedge t)|^2 + 2\alpha \E\int_0^t
  \E|y(\tau_n \wedge s)|^2\,ds
  \]
  Appealing to Gronwall's inequality and recalling that $\tau_n \to
  \infty$ as $n \to \infty$, we obtain $y(s)=u(s)-v(s)=0$
  $\P$-a.s. for all $s \leq T$, thus completing the proof.
\end{proof}
\begin{rmk}
  Note that uniqueness in the previous theorem is obtained in the
  class of solutions satisfying (\ref{eq:strigne}), which is stronger
  than necessary for existence. On the other hand, since stochastic
  convolutions of pseudo-contraction semigroups with respect to
  general (locally square integrable) martingales are c\`adl\`ag (see
  e.g. \cite{Kote-Doob}), it follows that solutions to (\ref{eq:zero})
  will also be c\`adl\`ag.
\end{rmk}


\section{Uniqueness of weak solutions}
In this section we prove that (\ref{eq:zero}) admits a unique weak
solution. By weak solution we shall always mean weak solution in the
analytic sense, not in the probabilistic sense. In particular, we
shall say that $u$ is a weak solution of (\ref{eq:zero}) if
\begin{multline*}
\ip{u(t)}{\phi} + \int_0^t \ip{u(s)}{A^*\phi}\,ds
+ \int_0^t \ip{Fu(s)}{\phi}\,ds\\
= \Big\langle \int_0^t B(s,u(s))\,dW(s), \phi \Big\rangle
+ \int_0^t\!\int_Z \ip{G(s,u(s-),z)}{\phi}\,\bar{\mu}(ds,dz)
\end{multline*}
for all $\phi \in D(A^*)$. Here and in the following $A^*$ stands for
the adjoint of $A$.

The assumptions of Theorem \ref{thm:boot} will be in force throughout
this section.
\begin{thm}
  The stochastic equation (\ref{eq:zero}) admits at most one
  c\`adl\`ag weak solution satisfying the integrability condition
  (\ref{eq:strigne}).
\end{thm}
\begin{proof}
  Let $u$ and $v$ be two weak solutions of (\ref{eq:zero}), and set
  $\tilde{e}_k:=(I+\varepsilon A^*)^{-1} e_k$, where
  $\{e_k\}_{k\in\mathbb{N}}$ is a complete orthonormal basis of
  $H$. We can then write (all claims will meant to hold $\P$-a.s.,
  if not otherwise stated)
  \begin{align*}
  \bip{\tilde{e}_k}{u(t)-v(t)} &+
  \int_0^t \bip{A^*\tilde{e}_k}{u(s)-v(s)}\,ds
  + \int_0^t \bip{\tilde{e}_k}{Fu(s)-Fv(s)}\,ds\\
  &= \int_0^t \bip{\tilde{e}_k}{\big(B(s,u(s))-B(s,v(s))\big)\,dW(s)}\\
  &\quad + \int_0^t\!\int_Z \bip{\tilde{e}_k}{G(s,u(s-),z)-G(s,v(s-),z)}
                   \,\bar{\mu}(ds,dz)\\
  &=: \psi_k(t)+\xi_k(t).
  \end{align*}
  Setting $\phi_k(t):=\bip{\tilde{e}_k}{u(t)-v(t)}$, It\^o's formula
  yields
  \begin{equation}     \label{eq:k}
  \phi_k(t)^2 = 2\int_0^t \phi_k(s-)\,d\phi_k(s) + [\phi_k](t),
  \end{equation}
  where
  \begin{align*}
  \int_0^t \phi_k(s-)\,d\phi_k(s) &=
  - \int_0^t \bip{A^*\tilde{e}_k}{u(s)-v(s)} \bip{\tilde{e}_k}{u(s)-v(s)}\,ds\\
  &\quad - \int_0^t \bip{\tilde{e}_k}{Fu(s)-Fv(s)}
  \bip{\tilde{e}_k}{u(s)-v(s)}\,ds\\
  &\quad + \int_0^t \phi_k(s)\,d\psi_k(s) + \int_0^t
  \phi_k(s-)\,d\xi_k(s),
  \end{align*}
  and
\begin{gather*}
  [\phi_k](t) = I_k^1(t) + I_k^2(t),\\
  \begin{split}
  I_k^1(t) &= \Big[ \int_0^\cdot
              \bip{\tilde{e}_k}{\big(B(s,u(s))-B(s,v(s))\big)\,dW(s)}
              \Big](t)\\
           &= \Big[ \int_0^\cdot
              \bip{\big(B(s,u(s))^*-B(s,v(s))^*\big)\tilde{e}_k}{dW(s)}
              \Big](t)\\
           &= \int_0^t \big| (I+\varepsilon A)^{-1}
              \big(B(s,u(s))-B(s,v(s))\big)e_k \big|^2\,ds,\\
  I_k^2(t) &= \Big[ \int_0^\cdot\!\int_Z \bip{\tilde{e}_k}
              {\big(G(s,u(s-),z)-G(s,v(s-),z)\big)}\bar{\mu}(ds,dz)
              \Big](t)\\
           &= \int_0^t\!\int_Z \big| (I+\varepsilon A)^{-1}
              \big(G(s,u(s-),z)-G(s,v(s-),z)\big)e_k \big|^2\,\mu(ds,dz).
  \end{split}
\end{gather*}
Note that since $A^*$ and $(I+\varepsilon A^*)^{-1}$ commute, and
$(I+\varepsilon A^*)^{-1}=((I+\varepsilon A)^{-1})^*$, the dominated
convergence theorem yields the following relations:
\begin{align*}
&\sum_{k\leq N} \int_0^t \bip{A^*\tilde{e}_k}{u(s)-v(s)} \bip{\tilde{e}_k}{u(s)-v(s)}\,ds\\
&\hspace*{5em} = \int_0^t \sum_{k\leq N}
                 \bip{e_k}{A(I+\varepsilon A)^{-1} (u(s)-v(s))}
                    \bip{e_k}{(I+\varepsilon A)^{-1} (u(s)-v(s))}\,ds\\
&\hspace*{5em} \xrightarrow{N \to \infty} 
               \int_0^t \bip{A(I+\varepsilon A)^{-1}(u(s)-v(s))}{%
                    (I+\varepsilon A)^{-1}(u(s)-v(s))}ds,
\end{align*}
and
\begin{align*}
&\sum_{k\leq N} \int_0^t \bip{\tilde{e}_k}{Fu(s)-Fv(s)}
         \bip{\tilde{e}_k}{u(s)-v(s)}\,ds\\
&\hspace*{5em} = \int_0^t \sum_{k\leq N}
         \bip{e_k}{(I+\varepsilon A)^{-1}(Fu(s)-Fv(s))}
           \bip{e_k}{(I+\varepsilon A)^{-1}u(s)-v(s)}\,ds\\
&\hspace*{5em} \xrightarrow{N \to \infty} 
         \int_0^t \bip{(I+\varepsilon A)^{-1}(Fu(s)-Fv(s))}{%
                           (I+\varepsilon A)^{-1}(u(s)-v(s))}ds.
\end{align*}
In fact, one has
\[
\bip{A(I+\varepsilon A)^{-1}(u(s)-v(s))}{%
                    (I+\varepsilon A)^{-1}(u(s)-v(s))}
\lesssim_\varepsilon |u(s)-v(s)|^2
\]
and
\[
\bip{(I+\varepsilon A)^{-1}(Fu(s)-Fv(s))}{%
                           (I+\varepsilon A)^{-1}(u(s)-v(s))}
\leq |Fu(s)-Fv(s)| \, |u(s)-v(s)|,
\]
which imply the claim recalling that $\sup_{t\leq T}
|u(t)-v(t)|<\infty$ and $Fu$, $Fv \in L^1([0,T] \to H)$
$\P$-a.s..
Furthermore, let us define, for each $s \leq T$, the operators $C(s): K
\to \erre$,
\[
C(s): \zeta \mapsto \ip{(I+\varepsilon A)^{-1}u(s)-v(s)}
                       {(I+\varepsilon A)^{-1}(B(s,u(s))-B(s,v(s)))\zeta},
\]
and $C_N(s): K \to \erre$,
\[
C_N(s): \zeta \mapsto \sum_{k\leq N} \ip{\tilde{e}_k}{u(s)-v(s)}
\ip{\tilde{e}_k}{(B(s,u(s))-B(s,v(s))\zeta}.
\]
Then it is clear that $C_N(s) \to C(s)$ in probability as $N \to
\infty$ for all $s\leq T$, and
\[
|C_N(s)|_Q \leq |u(s)-v(s)| \, \big|B(u(s)-B(s,v(s))\big|_Q,
\]
\begin{multline*}
\int_0^T |u(s)-v(s)|^2 \, \big|B(u(s)-B(s,v(s))\big|^2_Q\,ds\\
\leq \sup_{s\leq T} |u(s)-v(s)|^2 \int_0^T \big|B(u(s)-B(s,v(s))\big|^2_Q\,ds
< \infty,
\end{multline*}
which implies that $(C_N \cdot W)_t \to (C \cdot W)_t$ in
probability as $N \to \infty$ for all $t \leq T$, or equivalently
\begin{multline*}
\sum_{k\leq N} \int_0^t \phi_k(s)\,d\psi_k(s)\\
\xrightarrow{N \to \infty} M^1_\varepsilon(t):=
\int_0^t \bip{(I+\varepsilon A)^{-1}(u(s)-v(s))}
{(I+\varepsilon A)^{-1}\big(B(s,u(s))-B(s,v(s))\big)\,dW(s)}
\end{multline*}
in probability for all $t \leq T$. An analogous reasoning yields
\begin{gather*}
\sum_{k\leq N} \int_0^t \phi_k(s)\,d\xi_k(s)
\xrightarrow{N \to \infty} M^2_\varepsilon(t),\\
M^2_\varepsilon(t) := 
\int_0^t\!\int_Z \bip{(I+\varepsilon A)^{-1}(u(s)-v(s))}
          {(I+\varepsilon A)^{-1}\big(G(s,u(s-),z)-G(s,v(s-),z)\big)}\,\m(ds,dz)
\end{gather*}
in probability for all $t \leq T$.
Finally, the following obvious inequalities hold:
\begin{align*}
\sum_{k \leq N} I^1_k(t) &\leq \int_0^t \big|
(I+\varepsilon A)^{-1}\big(B(s,u(s))-B(s,v(s))\big) \big|_Q^2\,ds,\\
\sum_{k \leq N} I^2_k(t) &\leq \int_0^t\!\int_Z \big|
(I+\varepsilon A)^{-1}\big(G(s,u(s-),z)-G(s,v(s-),z)\big)\big|^2\,
\mu(ds,dz)
\end{align*}
for all $N$.

Summing up over $k \leq N$ in (\ref{eq:k}) and letting $N \to \infty$
yields
\begin{align*}
& \big| (I+\varepsilon A)^{-1}(u(t)-v(t)) \big|^2\\
&\hspace*{4em} + 2 \int_0^t \bip{A(I+\varepsilon A)^{-1}(u(s)-v(s))}{%
                    (I+\varepsilon A)^{-1}(u(s)-v(s))}ds\\
&\hspace*{4em} + 2 \int_0^t \bip{(I+\varepsilon A)^{-1}(Fu(s)-Fv(s))}{%
                           (I+\varepsilon A)^{-1}(u(s)-v(s))}ds\\
&\qquad \leq M_\varepsilon(t) + \int_0^t \big| (I+\varepsilon A)^{-1}
              \big(B(s,u(s))-B(s,v(s))\big) \big|_Q^2\,ds\\
&\hspace*{4em} + \int_0^t\!\int_Z \big| (I+\varepsilon A)^{-1}
                  \big(G(s,u(s-),z)-G(s,v(s-),z)\big) \big|^2\,\mu(ds,dz)
\end{align*}
where $M_\varepsilon:=M^1_\varepsilon + M^2_\varepsilon$ is a local
martingale. By the monotonicity of $A$, the previous inequality yields
\begin{align*}
& \big| (I+\varepsilon A)^{-1}(u(t)-v(t)) \big|^2\\
&\hspace*{4em} + 2 \int_0^t \bip{(I+\varepsilon A)^{-1}(Fu(s)-Fv(s))}{%
                           (I+\varepsilon A)^{-1}(u(s)-v(s))}ds\\
&\qquad \leq M_\varepsilon(t) + \int_0^t \big| 
              \big(B(s,u(s))-B(s,v(s))\big) \big|_Q^2\,ds\\
&\hspace*{4em} + \int_0^t\!\int_Z \big| 
                  \big(G(s,u(s-),z)-G(s,v(s-),z)\big) \big|^2\,\mu(ds,dz).
\end{align*}
We are now going to pass to the limit as $\varepsilon \to 0$ in the
inequality just obtained. Trivially, the first-term on the right hand
side converges to $|u(t)-v(t)|^2$, while the second term on the
left-hand side converges to
\[
2 \int_0^t \ip{Fu(s)-Fv(s)}{u(s)-v(s)}\,ds
\]
by the dominated convergence, in analogy to a situation already
entountered. The contractivity of $(I+\varepsilon A)^{-1}$ also
implies $M_\varepsilon(t) \to M(t)$ as $\varepsilon \to 0$ in
probability for all $t$, where $M$ is the same local martingale
defined in the proof of Theorem \ref{thm:boot}.  We are thus left with
\begin{multline*}
|u(t)-v(t)|^2 + 2 \int_0^t \ip{Fu(s)-Fv(s)}{u(s)-v(s)}\,ds\\
\leq M(t) + \int_0^t \big| 
              \big(B(s,u(s))-B(s,v(s))\big) \big|_Q^2\,ds
+ \int_0^t\!\int_Z \big| 
                  \big(G(z,u(s))-G(z,v(s))\big) \big|^2\,\mu(ds,dz),
\end{multline*}
and the proof is completed exactly as in the previous section,
i.e. taking a sequence of localizing stopping times for $M$, etc.
\end{proof}

\begin{rmk}
  Using a stochastic Fubini theorem in infinite dimensions (see
  e.g. \cite{Leon-Fub}), it is not difficult to see that weak and mild
  solutions of (\ref{eq:zero}) coincide, provided the integrability
  condition (\ref{eq:strigne}) is satisfied (cf. \cite{choj76}).
\end{rmk}


\section{Uniqueness of generalized solutions}
The purpose of this section is to show that, in certain cases, one can
still prove uniqueness for equations whose solutions $u$ do not
satisfy the integrability condition $Fu \in L^1([0,T]\to H)$
$\P$-a.s.. In fact, in general it is difficult (and we are not aware
of any general results or techniques) to prove well-posedness in the
mild sense without imposing rather restrictive conditions on the
initial condition and on the coefficients of the equations. A possible
way out is to define ``generalized'' mild solutions as limits of
solutions of equations with more regular $u_0$, $B$, and $G$. Let us
make this notion precise. In the following we shall say that $\zeta
\in \mathcal{H}_2(T)$ if $\zeta:[0,T] \to H$ is an adapted process
such that $\sup_{t \leq T} \E|\zeta(t)|^2 < \infty$.
\begin{defi}
  Let
  \[
  \E|u_{0n} - u_0|^2 + \E\int_0^T \big(
  |B_n(t)-B(t)|_Q^2 + |G_n(t,\cdot)-G(t,\cdot)|_m^2
  \big)\,dt \to 0
  \]
  as $n \to \infty$, and assume that the equation
  \[
  du(t) + Au(t)\,dt + Fu(t)\,dt = B_n(t)\,dW(t) + G_n(t,z)\,\m(dt,dz)
  \]
  with initial condition $u(0)=u_{0n}$ admits a unique mild solution
  $u_n \in \mathcal{H}_2(T)$ for all $n \in \mathbb{N}$, such that
  $|u_n-u|_{\mathcal{H}_2(T)} \to 0$ as $n \to \infty$. Then $u$ is
  called a \emph{generalized mild} solution of (\ref{eq:add}).
\end{defi}
Unfortunately we cannot give general sufficient conditions ensuring
well-posedness of (\ref{eq:add}), but we limit ourselves to giving one
criterion which can be verified, for instance, for reaction-diffusion
equations with polynomial nonlinearity $F$, as considered in
e.g. \cite{DP-K} in the case of Wiener noise, and in \cite{cm:rd} in the
case of Poisson noise. In the latter reference one may also
find a fixed-point argument leading to existence and uniqueness of
generalized mild solutions for equations with multiplicative noise.

Uniqueness of generalized mild solutions can be obtained by a priori
estimates for mild solutions. For instance, let $u^1$, $u^2$ be
solutions of (\ref{eq:add}) with initial conditions $u^1_0$, $u^2_0$,
and coefficients $B^1$, $B^2$ and $G^1$, $G^2$, respectively. Assume
that the following estimate holds
\begin{multline*}
\E|u^1(t)-u^2(t)|^2 
\leq N \Big( \E|u^1_0 - u^2_0|^2
+ \E\int_0^t \big( 
  |B^1(s)-B^2(s)|_Q^2 + |G^1(s,\cdot)-G^2(s,\cdot)|_m^2\big) \,ds\Big), 
\end{multline*}
where the constant $N$ depends continuously on $t$. Since the
inequality is stable with respect to the limit passages of the
previous definition, it is immediate to see that the same estimate
holds also for generalized mild solution. This in turn implies that
the generalized mild solution, if it exists, is unique, simply by
taking $u^1_0=u^2_0$ and $G^1=G^2$.


\bibliographystyle{amsplain}
\bibliography{ref}

\end{document}